\documentclass[12pt]{article}
 
\usepackage[a4paper, width=170mm, top=35mm, bottom=35mm, bindingoffset=6mm]{geometry}

\usepackage{amssymb} %maths
\usepackage{amsmath} %maths
\usepackage{amssymb,bm,amsmath,amssymb,amsthm}
\usepackage[]{inputenc}
\usepackage[small,nohug,heads=vee]{diagrams}
\usepackage{tikz-cd}

\diagramstyle[labelstyle=\scriptstyle]

\usepackage{authblk}

\usepackage[all]{xy}
  \SelectTips{cm}{10}     % Use the arrowheads that agree with inline arrows.
  \everyxy={<2.5em,0em>:} % Set the unit length in xy diagrams.

\usepackage{hyperref}

\DeclareMathOperator{\Gal}{Gal}
\DeclareMathOperator{\Ind}{Ind}
\DeclareMathOperator{\Tr}{Tr}

\newtheorem{cor}{Corollary}

\newtheorem{rem}{Remark}
\newtheorem{lem}{Lemma}

\newtheorem{theorem}{Theorem}
\newcommand{\overbar}[1]{\mkern 1.5mu\overline{\mkern-1.5mu#1\mkern-1.5mu}\mkern 1.5mu}

\begin{document}
\title{A Note on Number Fields Sharing the List of Dedekind Zeta-Functions of Abelian Extensions with some Applications towards the Neukirch-Uchida Theorem.}
\author{Pavel Solomatin \\
\texttt{p.solomatin@math.leidenuniv.nl}
}
\affil{Leiden University, Mathematical Department,\\
Niels Bohrweg 1, 2333 CA Leiden}

\date{ Den Haag, 2019}
\maketitle

\begin{abstract}
Given a number field $K$ one associates to it the set $\Lambda_K$ of Dedekind zeta-functions of finite abelian extensions of $K$. In this short note we present a proof of the following Theorem: for any number field $K$ the set $\Lambda_K$ determines the isomorphism class of $K$. This means that if for any number field $K'$ the two sets $\Lambda_K$ and $\Lambda_{K'}$ coincide, then $K \simeq K'$. As a consequence of this fact we deduce an alternative approach towards the proof of Neukirch-Uchida Theorem for the case of non-normal extensions of number fields.  
\end{abstract}

\textbf{\\ \\ \\ \\ \\ \\ \\  \\ \\ \\ Acknowledgements:} 
I would like to express my gratitude to professor Bart de Smit for all of his efforts and exciting discussions during the project. 

\newpage

\section{Introduction}
\subsection{Motivation}
Let $K$ be a number field, i.e.\ a finite extension of the field of rational numbers $\mathbb Q$ and let $\mathcal G_K$ be the absolute Galois group of $K$, that is $\mathcal G_K = \Gal( \overbar{ K} : K)$ for some fixed algebraic closure $\overbar{K}$ of $K$. The isomorphism class of $\mathcal G_K $ captures a lot of important information about arithmetical properties of $K$. The famous Neukirch-Uchida Theorem states that for given number fields $K, K'$ the existence of a topological isomorphism of pro-finite groups $\mathcal G_K \simeq \mathcal G_{K'}$ implies the existence of an isomorphism of fields $K~\simeq~K'$ themselves. In 1969 Neukirch \cite{NeukAnotherProof} gave a proof for the case of normal extensions of $\mathbb Q$. He proved this by recovering Dedekind zeta-function $\zeta_K(s)$ of $K$ from $\mathcal G_K$ in group-theoretical terms and then applying the famous \emph{Chebotarev density} argument to show that in this case $\zeta_K(s)$ determines the isomorphism class of $K$. A few years later in 1976 Uchida \cite{Uchida1} extended Neukirch's results to the case of arbitrary number fields. Uchida's approach was then also used by himself and others to generalise the Theorem to the case of all global fields. For a modern exposition see Chapter XII in \cite{neukirchCohomology}. Without any doubt Uchida's proof is beautiful and important, but it contains some difficult technical details which make this proof a bit less clear especially for those who are relatively new to the topic. The goal of the present note is to provide an alternative, in some sense more elementary approach to the proof of Uchida's part. The new proof also has another advantage, since it stays closer to Neukirch's original idea. This new approach is based on the following idea. 

Given a number field $K$ we associate to it a set $\Lambda_K$ of Dedekind zeta-functions of finite abelian extensions of $K$: $$ \Lambda_K = \{ \zeta_L(s) \ | \ L \text{ is a finite abelian extension of } K \}.$$ Our main goal is to prove the following Theorem: 

\begin{theorem}\label{ZetaFunctionMain}
For any number field $K$ the set $\Lambda_K$ determines the isomorphism class of~$K$. This means that if for any other number field $K'$ the two sets $\Lambda_K$ and $\Lambda_{K'}$ coincide, then $K \simeq K'$. 
\end{theorem} 

The following observation shows that Theorem \ref{ZetaFunctionMain} allows us to achieve our goal and produce an alternative way to Uchida's part:

\begin{cor}
In the above settings suppose that $\mathcal G_K \simeq \mathcal G_{K'}$. Then $\Lambda_K = \Lambda_{K'}$ and therefore $K \simeq K'$. 
\end{cor}
\begin{proof}
Indeed, given an isomorphism class of $\mathcal G_K$ we consider all closed subgroups of finite index $H \subset \mathcal G_K$ such that the quotient $\mathcal G_K / H$ is a finite abelian group. By pro-finite Galois theory we have one-to-one correspondence between such $H$ and finite abelian extensions $L$ of $K$ within fixed algebraic closure $\overbar{K}$ given by $H \to (\overbar{K})^{H}$. Now by using Neukirch's Theorem, see chapter 4 in \cite{NeukAnotherProof} we reconstruct $\zeta_L(s)$ in a group theoretical manner from $H$ and therefore reconstruct $\Lambda_K$ from $\mathcal G_K$.
\end{proof}

From now on we concentrate our attention towards the proof of \ref{ZetaFunctionMain}.

\subsection{On the Proof of Theorem \ref{ZetaFunctionMain}}

The inspiration behind Theorem \ref{ZetaFunctionMain} is the following result due professor Bart de Smit:

\begin{theorem}{\label{BartMain}}
For each number field $K$ there exists an abelian extension $N_K$ of degree three and a character $\chi$ of $\Gal(N_K : K)$ such that $L_K(\chi,s)$ occurs only for the isomorphism class of the field $K$, i.e. if for any other number field $K'$ and any abelian extension $N_{K'}$ of $K'$ there exists a character $\chi'$ of $\Gal(N_{K'}: K')$ such that $L_K(\chi,s) = L_{K'}(\chi',s)$ then $K\simeq K'$.  
\end{theorem} 
\begin{proof}
See Theorem 10.1 from \cite{GlobalFieldsFromLFunct}.
\end{proof}

To deduce Theorem \ref{ZetaFunctionMain} we extend Theorem \ref{BartMain} by replacing the L-function of the abelian character $\chi$ by the Dedekind $\zeta$-function of the abelian extension $N_K$ of $K$:

\begin{theorem}{\label{MainArticle}}
For each number field $K$ there exists an abelian extension $N_K$ of degree three such that pair $\zeta_{N_K} (s)$, $\zeta_K (s)$ occurs only for the isomorphism class of the field $K$, i.e. if for any other number field $K'$ and any abelian extension $N_{K'}$ of $K'$ we have $\zeta_K(s) = \zeta_{K'}(s)$ and $\zeta_{N_K} (s)$~=~$\zeta_{{N'}_{K'}} (s)$  then $K\simeq K'$.  
\end{theorem} 

\begin{rem}\label{remDegreeNumberField}
Note that the degree of a number field $K$ is determined by $\zeta_K(s)$. Therefore, $\zeta_K(s)$ can be recovered from $\Lambda(K)$ as unique element whose corresponding field has minimal degree.
\end{rem}

The above remark shows that Theorem \ref{ZetaFunctionMain} and Theorem \ref{MainArticle} are equivalent and we can focus on proving the last statement.

\begin{bf}The article has the following structure.\end{bf} In the next section we recall some basic group theoretical settings of Theorem \ref{BartMain} and then provide a proof for a generalisation of these settings, see Theorem \ref{GroupNew}. After that we will deduce Theorem \ref{MainArticle} from Theorem \ref{GroupNew}. 

\section{The proof}
\subsection{Group Theoretical Preliminaries}
Let $G$ be a finite group, $H$ a subgroup of index $n$, and $C_l = \mu_l \subset \mathbb C^{\times}$ be a cyclic group of order $l$, where \emph{$l$ is an odd prime}. Consider the $G$-set $G/H$ of left cosets. We fix some representatives $X_1, \dots X_n$ of $G/H$ such that $X_1$ is a coset corresponding to the group $H$. Let us regard semi-direct products $\tilde{G}=C_l^{n}\rtimes G$ and $\tilde{H}=C_l^{n}\rtimes H$, where $G$ acts on the components of $C_l^{n}$ by permuting them as the cosets $\{X_1, \dots ,X_n \}$. Let $\chi$ be the homomorphism from $\tilde{H}$ to the group $C_l$ defined on the element $\tilde{h} = (\zeta_1, \dots, \zeta_n, h) \in \tilde{H} = C_l^{n}\rtimes H$ as $\chi(\tilde{h}) = \zeta_1$ i.e. $\chi$ is the \emph{projection on the first coordinate}. This is indeed a homomorphism because every $h \in H$ fixes the first coset of $G / H$. In this setting the following holds:
 
\begin{theorem}[Bart de Smit]\label{Group}
For any  subgroup $\tilde{H}' \subset \tilde{G}$ and any character $\chi'$ : $\Tilde{H}' \to \mathbb C^{*}$ if $\Ind_{\tilde{H}'}^{\tilde{G}} (\chi') \simeq  \Ind_{\tilde{H} }^{\tilde{G}} (\chi ) $ then $\tilde{H'}$ and $\tilde H$ are conjugate in $\tilde{G}$.
\end{theorem}
\begin{proof}
See Theorem 7 from \cite{ArEq}.
\end{proof}

To apply this result we fix $l = 3$ and first prove the following auxiliary statement:
\begin{lem}\label{irredLem}
The induced representation $\Ind_{\tilde{H} }^{\tilde{G}} (\chi)$ is an irreducible representation of $\tilde{G}$.
\end{lem}  
\begin{proof}

In order to verify irreducibility of $\Ind_{\tilde{H} }^{\tilde{G}} (\chi)$ we regard the standard scalar product and show that $ (\Ind_{\tilde{H} }^{\tilde{G}} (\chi), \Ind_{\tilde{H} }^{\tilde{G}} (\chi))_{\tilde{G}}  = 1 $. Applying Frobenius reciprocity: 

\begin{equation}\label{ExprChar}
 (\Ind_{\tilde{H} }^{\tilde{G}} (\chi), \Ind_{\tilde{H} }^{\tilde{G}} (\chi))_{\tilde{G}} =  ( \chi, \Ind_{\tilde{H} }^{\tilde{G}} (\chi)|_{\tilde{H}} )_{\tilde{H}} = \frac{1}{| \tilde{H} |} \sum_{\tilde{h} \in \tilde{H}} \bar{\chi}(\tilde{h}) \cdot  \Tr ( \Ind_{\tilde{H} }^{\tilde{G}} (\chi)|_{\tilde{H}} (\tilde{h})). 
\end{equation}

Let $ \tilde{h} = (\zeta_1, \dots, \zeta_n, h) \in \tilde H$ then by definition of $\chi$ we have $\bar{\chi}(\tilde{h}) = \bar{\zeta_1}$. Now consider the matrix $\Ind_{\tilde{H} }^{\tilde{G}} (\chi)|_{\tilde{H}} (\tilde{h})$. We fix the following representatives for cosets of $\tilde{G} / \tilde{H}$ as $\tilde{X_i} = (1, \dots, 1, X_i) \in \tilde{G}$, where $X_i$ are representative of cosets of $G/H$ we picked before. By definition of the induced representation and because $h$ fixes first conjugacy class of $G/H$ we have that in the top left corner of that matrix $\zeta_{1}$ is located. Now we fix an integer $1< i \le n$ and consider diagonal element $a_{i}( \tilde{h})$ on $(i, i)$-th position. Regard the permutation of cosets $\tilde{X_1}, \dots \tilde{X_n}$ by $\tilde{h}$ and denote by $j$ an index such that $\tilde{h} \tilde{X_i} = \tilde{X_j}$. If $i \ne j$ then $a_{i}( \tilde{h}) = 0$ and therefore such $i$ adds no contribution to the expression \ref{ExprChar}. Otherwise, by definition of the induced representation $\Ind_{\tilde{H} }^{\tilde{G}} (\chi)|_{\tilde{H}} $ we have $a_{i}( \tilde{h}) = \chi ( \tilde{X}_i^{-1} \tilde{h} \tilde{X}_i ) = \zeta_{k_i}$ for some index $k_i \in \{2, \dots, n \}$. In other words, $k_i$ is an index such that $X_i^{-1} X_{k_i} \in H$. For fixed $\tilde{h}$ and $i$ there are elements $\tilde{h}_1$, $\tilde{h}_{2}$ such that $(\tilde{h}, \tilde{h}_1, \tilde{h}_{2} )$ pairwise coincide in all coordinates except the $k_i$-th one. Because $1+\zeta_{k_i} + \bar{\zeta}_{k_i} = 0$ we have that sum of $a_i(\tilde{h})$ for those $\tilde{h}$, $\tilde{h}_1$, $\tilde{h}_2$ is zero and because they coincide on first coordinate we have $\chi(\tilde{h}) = \chi(\tilde{h}_j)$ for $j$ in $\{1, 2 \}$. Therefore for fixed $i>1$ we have: $$\sum_{\tilde{h} \in \tilde{H} } \bar{\chi}(\tilde{h}) a_i (\tilde{h}) = 0.$$

Now we regard expression \ref{ExprChar}: 
\begin{align*} 
\frac{1}{| \tilde{H} |} \sum_{\tilde{h} \in \tilde{H}} \bar{\chi}(\tilde{h}) \cdot  \Tr ( \Ind_{\tilde{H} }^{\tilde{G}} (\chi)|_{\tilde{H}} (\tilde{h})) = \frac{1}{| \tilde{H} |} \sum_{\tilde{h} \in \tilde{H}} \bar{\chi}(\tilde{h}) \cdot ( \chi(\tilde{h}) + \sum_{i > 1}^{i \le n} a_i (\tilde{h})) = \\
= \frac{1}{| \tilde{H} |} \sum_{\tilde{h} \in \tilde{H}} \bar{\chi}(\tilde{h}) \chi(\tilde{h}) + \frac{1}{| \tilde{H} |} \sum_{\tilde{h} \in \tilde{H}} ( \bar{\chi}(\tilde{h}) \cdot ( \sum_{i > 1}^{i \le n} a_i (\tilde{h}) )) = \\
= \frac{1}{| \tilde{H} |} \sum_{\tilde{h} \in \tilde{H}} 1 + \frac{1}{| \tilde{H} |} \sum_{i>1} ( \sum_{\tilde{h} \in \tilde{H}} \bar{\chi}(\tilde{h}) a_i (\tilde{h})  )= 1 + 0.
\end{align*}
\end{proof}

By using this lemma we can prove the main group theoretical result of this note:

\begin{theorem}\label{GroupNew}
In the above settings let $U_{\tilde{H}, \chi} = \ker( \chi ) = \{ h \in \tilde{H} | \chi(h) = 1 \}$ and let $U_{\tilde{H'}, \chi'} = \ker( \chi')$. Suppose that $ \Ind_{\tilde{H}}^{\tilde{G}}(1) \simeq \Ind_{\tilde{H'}}^{\tilde{G}}(1)$ and $ \Ind_{U_{\tilde{H}, \chi}}^{\tilde{G}} (1) \simeq \Ind_{U_{\tilde{H'}, \chi'}}^{\tilde{G}} (1)$. Then either $ \Ind_{\tilde{H} }^{\tilde{G}} (\chi ) \simeq \Ind_{\tilde{H}'}^{\tilde{G}} (\chi')$ or $\Ind_{\tilde{H} }^{\tilde{G}} (\chi) \simeq \Ind_{\tilde{H}'}^{\tilde{G}} (\bar{\chi}')$.  
\end{theorem}
\begin{proof}
Since $l=3$ we have that $C_l$ has only three characters $1, \chi, \bar{\chi}$ and therefore: $$ \Ind_{U_{\tilde{H}, \chi}}^{\tilde{G}} (1) \simeq \Ind_{\tilde{H} }^{\tilde{G}} (\chi) \oplus \Ind_{\tilde{H} }^{\tilde{G}} (\bar{\chi}) \oplus \Ind_{\tilde{H}}^{\tilde{G}}(1).$$ Hence, from the assumption of the Theorem it follows that: 
$$ \Ind_{\tilde{H} }^{\tilde{G}} (\chi) \oplus \Ind_{\tilde{H} }^{\tilde{G}} (\bar{\chi}) \simeq \Ind_{\tilde{H'} }^{\tilde{G}} (\chi') \oplus \Ind_{\tilde{H'} }^{\tilde{G}} (\bar{\chi'}).$$

In lemma \ref{irredLem} we showed that $\Ind_{\tilde{H} }^{\tilde{G}} (\chi)$, $\Ind_{\tilde{H} }^{\tilde{G}} (\bar{\chi})$ are \emph{irreducible representation} of $\tilde{G}$. But if a direct sum of two irreducible representations of a finite group is isomorphic to a direct sum of two other \emph{non-zero representations} then those representations are pairwise isomorphic up to swap. It follows that either $ \Ind_{\tilde{H} }^{\tilde{G}} (\chi ) \simeq \Ind_{\tilde{H}'}^{\tilde{G}} (\chi')$ or $\Ind_{\tilde{H} }^{\tilde{G}} (\chi)~\simeq~\Ind_{\tilde{H}'}^{\tilde{G}} (\bar{\chi}')$.
\end{proof}

\section{Recovering Number Fields from $\zeta$-functions}
\begin{proof}[Proof of Theorem \ref{MainArticle}]
Suppose $K$ is a number field such that $\zeta_K(s)$ does not determine $K$ i.e. there exists a number field $K'$ such that $\zeta_K(s) = \zeta_{K'}(s)$, but $K \not \simeq K'$. Then according to the well-known Theorem of Perlis \cite{Perl} this means that the normal closure $N$ of $K$ contains $K'$ and there exists a \emph{non-trivial Gassmann triple} $(G,H,H')$ with $G= \Gal(N / \mathbb Q) $, $H~=~\Gal(N / K) $, $H' = \Gal(N / K') $.  

In this settings we construct a Galois extension $M$ of $Q$ containing $K$ and $K'$ such that the Galois group $\Gal(M : \mathbb Q)$ is $\tilde{G}$ and $K=M^{\tilde{H}}$, $K'=M^{\tilde{H'}}$ for $\tilde{G}$, $\tilde{H}$,  $\tilde{H'}$ as in Theorem \ref{Group}. This is possible due to Proposition 9.1 from \cite{GlobalFieldsFromLFunct}. See the diagram below:

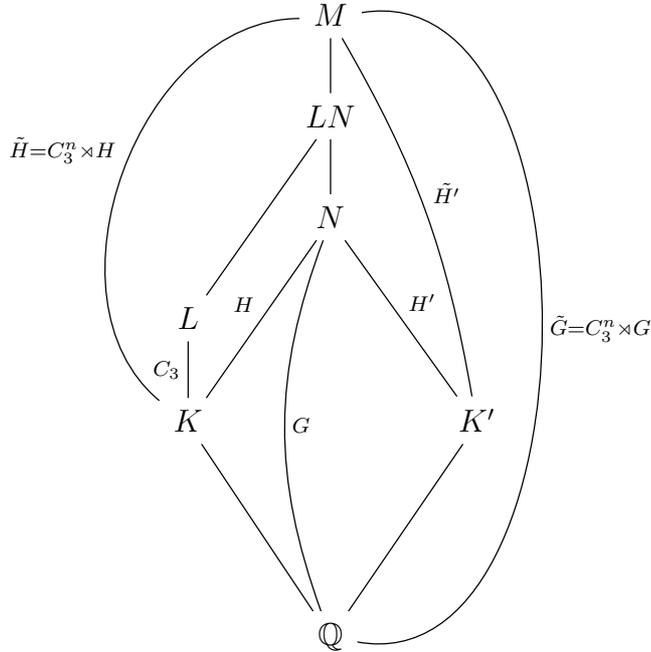
\begin{center}
\begin{tikzcd} 
 &  M  \arrow[rdddd, dash,bend left = 10, "\tilde{H'}"] \arrow[ldddd, dash,bend right = 70, "\tilde{H}=C_{3}^{n} \rtimes H"'] \arrow[ddddddd, dash,bend left = 100, "\tilde{G} = C_{3}^{n} \rtimes G"]   \arrow[d, dash] \\	
 & LN \arrow[d, dash] \arrow[ldd, dash ]  \\ 
 & N \arrow[ldd, dash, "H"' ] \arrow[rdd, dash, "H'"]  \arrow[ddddd, dash,bend right = 20, "G"] \\ 
 L\arrow[d, dash, "C_3"' ]  &   \\	
  K \arrow[rddd,dash] & &K' \arrow[lddd,dash]  \\ \\	 \\
  & \mathbb Q 
\end{tikzcd}
\end{center}

Our goal is to show that $\zeta_L(s)$ occurs only for fields isomorphic to $K$. Let $L'$ be \emph{any abelian extension} of $K'$ such that $\zeta_{L'} = \zeta_{L}$. Then $L$ and $L'$ share the same normal closure over $\mathbb Q$ and therefore $L'$ is a subfield of $M$. According to remark \ref{remDegreeNumberField} we also have that degree of $L'$ over $K'$ is three. Observe that in notations of Theorem \ref{GroupNew} from the previous section one has: $ \Gal(M : L) = \ker(\chi) = U_{\tilde{H}, \chi}$ for a non-trivial character $\chi$ of $\Gal(L:K)$ and $ \Gal(M : L') = \ker(\chi')$ for a non-trivial character $\chi'$ of $\Gal(L':K')$. By the induction property of Artin L-functions we have: $ \zeta_L(s) = L_{\mathbb Q}( \Ind_{U_{\tilde{H}, \chi}}^{\tilde{G}} (1), s)$.

Finally, because of Chebotarev density argument and because every place of $\mathbb Q$ is determined by the characteristic of its residue degree: 

$$ L_{\mathbb Q}( \Ind_{U_{\tilde{H}, \chi}}^{\tilde{G}} (1), s) = L_{\mathbb Q}( \Ind_{U_{\tilde{H'}, \chi'}}^{\tilde{G}} (1), s) \Leftrightarrow  \Ind_{U_{\tilde{H}, \chi}}^{\tilde{G}} (1) \simeq \Ind_{U_{\tilde{H'}, \chi'}}^{\tilde{G}} (1).$$

This means that from assumption of Theorem \ref{MainArticle} we deduced conditions of Theorem \ref{GroupNew}. Therefore because of Theorem \ref{Group} we have that $\tilde{H}$ and $\tilde{H'}$ are conjugate and hence $K$ is isomorphic to $K'$. 

\end{proof}

\newpage

\bibliography{mybib}{}
\bibliographystyle{plain}

\newpage

\tableofcontents

\end{document}